\documentclass[12pt,oneside,a4paper]{amsart}
\usepackage[all]{xypic}
\usepackage{amssymb}
\usepackage{amsmath}
\usepackage{enumerate}
\usepackage{stackrel}
\usepackage{a4wide}

\newcommand{\strong}{\textbf}

\newcommand{\Hom}{\operatorname{Hom}}
\newcommand{\Ext}{\operatorname{Ext}}

\newcommand{\id}{\operatorname{id}}
\newcommand{\Id}{\mathrm{Id}}
\newcommand{\Aut}{\operatorname{Aut}}

\newcommand{\lhom}{\operatorname{\mathcal{H}\mathit{om}}}

\newcommand{\rhom}{\operatorname{R\mathcal{H}\mathit{om}}}
\newcommand{\lra}{\longrightarrow}

\newcommand{\ltensor}{\mathbin{\smash{\buildrel L\over\otimes}}}
\newcommand{\Ltensor}{\mathbin{{\buildrel L\over\otimes}}}

\newcommand{\compo}{\raise2pt\hbox{$\scriptscriptstyle\circ$}}

\newcommand{\M}{\mathcal{M}}
\newcommand{\OO}{\mathcal{O}}

\newtheorem{thm}{Theorem}[section]

\newtheorem{lemma}[thm]{Lemma}
\newtheorem{prop}[thm]{Proposition}

\theoremstyle{definition}
\newtheorem{defn}[thm]{Definition}

\newtheorem{remark}[thm]{Remark}

\newcommand{\Pic}{\operatorname{Pic}}
\newcommand{\rk}{\operatorname{rk}}
\newcommand{\C}{\mathbf{C}}
\newcommand{\Hilb}{\operatorname{Hilb}}

\newcommand{\EE}{\mathbb{E}}

\newcommand{\ZZ}{\mathbb{Z}}

\title{A Note on Commuting Reflection Functors for Calabi-Yau d-folds}
\author{Antony Maciocia}

\address{Department of Mathematics and Statistics\\
The University of Edinburgh\\
The King's Buildings\\ Mayfield Road\\ Edinburgh, EH9 3JZ, UK.}
\email{A.Maciocia@ed.ac.uk}

\date{\today}
\begin{document}
\begin{abstract}
We study sets of commuting reflection functors in the derived category
of sheaves on Calabi-Yau varieties. We show that such a collection
is determined by a set of mutually orthogonal spherical objects. We
also show that when the spherical objects are locally-free sheaves
then the kernel of the composite transform parametrizes properly torsion-free
with zero-dimensional singularity sets and conversely that such a
kernel gives rise to a collection of mutually orthogonal spherical
vector bundles. We do this using a more detailed analysis of the
reason why spherical twists give equivalences.
\end{abstract}

\maketitle

\section*{Introduction}
A spherical object $e$ in an exact linear triangulated category $(T,[\ ])$ is one such
that $\dim\Hom(e,e[i])=\dim H^i(S^d,\mathbb{R})$, the Betti numbers of a
$d$-dimensional sphere for some fixed $d$. This concept is especially
useful when $T=D^b(X)$ for some
$d$-dimensional Calabi-Yau variety $X$ or when $T$ is a $d$-Calabi-Yau
category. This is because such objects have, in a suitable sense, the
fewest possible derived self-maps. There has been a great deal of
interest in them in recent years as they hold the key to understanding
the categorical structure of $T$ and its automorphism group $\Aut(T)$. For
example, it is conjectured that they give rise to a generating set for
$\Aut(T)$ in the case when $T=D(X)$, of a K3 surface. 
The spherical objects also play a central role in our understanding of
Bridgeland stability conditions for some surfaces (see \cite{BriK3})
essentially because of the central role they play in the derived
category of the surface. It is likely that they will play a similarly
crucial role in our understanding of stability conditions for higher
dimensional Calabi-Yau varieties. 

In an important paper, \cite{SeidelThomas}, Seidel and Thomas show
that certain series of spherical objects give rise to actions of the
braid group on the derived category. This is done by associating
an equivalence $\Phi_a$ of the derived category to each spherical object $a$ (known
as a spherical twist). In the K-theory of the derived category, these
are reflections and so they are sometimes called reflection
functors. It was observed in that paper that when two spherical
objects $a$ and $b$ are completely orthogonal (in other words,
$\Hom(a,b[i])=0$ for all integers $i$) then the associated spherical
twists commute. This is because
$\Phi_a\circ\Phi_b\cong \Phi_{\Phi_b(a)}\circ\Phi_b$ (see \cite[Lemma
2.11]{SeidelThomas}) and $\Phi_b(a)\cong a$ as can be checked by direct and
easy computation (see \cite[Proposition 2.12]{SeidelThomas}). Our
first result in this note is to show that the converse also holds: if
two spherical twists commute then either they are equal or the
associated spherical objects are (completely) orthogonal. 

We then turn our attention to the special case where the spherical
objects are actually vector bundles.  This is an important class of
examples. The associated spherical twists have Fourier-Mukai kernel
given by a sheaf parametrizing properly-torsion free sheaves whose
singularity set is a single point of $X$. Our second main result is to
show that this also has a converse: if $\Phi$ is an exact equivalence of
the derived categories of Calabi-Yau $d$-folds such that its
Fourier-Mukai kernel is a sheaf  parametrizing properly-torsion free sheaves whose
singularity set is zero dimensional then it must be a composite of
commuting spherical twists. The difficulty in this is to show that the
double dual of the kernel (which must be locally-free by assumption)
can be reduced essentially to a sum of (completely) orthogonal
spherical bundles. To establish this we need to generalise the
computations of $\Ext$ groups given by Mukai (\cite{MukTata}) and
which are so crucial in describing stability conditions for surfaces.

\section{Fourier-Mukai Transforms}\label{s:FMT}
In this paper we shall take a Fourier-Mukai transform (or FM transform
for short) to be an equivalence of categories of the (bounded) derived
category of sheaves $D(X)$ and $D(Y)$ on a smooth (complex) projective
varieties $X$ and  $Y$ given by correspondences of the form
$\Phi_F:E\mapsto Ry_*(x^*E\otimes F)$, where $F$ is a sheaf on
$X\times Y$ called the \strong{kernel} of the transform. These are
discussed in \cite{HuyBook} and \cite{BBHBook}.

Recall that we say that a family of sheaves $\M$ is
\strong{strongly simple} if it consists of simple sheaves and
if $\Ext^i(E,E')=0$ for all $i$ and $E\neq E'$ in $\M$. 
(see \cite{BrEquiv}):
\begin{thm}[Bridgeland]\label{t:TB} The kernel $F$ gives rise
to a FM transform if and only if the restrictions $F$ to $X$ form a
strongly simple family and $F_x\otimes K_X\cong F_x$ for all $F_x$ in
the family, where $K_X$ is the canonical bundles of $X$.
\end{thm}
The last condition is vacuous for Calabi-Yau $d$-folds.
The theorem gives us an easy way to recognise when a family of sheaves gives
rise to an FM transform.

We aim to study a special class of Fourier-Mukai Transforms which
arise from so called spherical bundles. These were first studied by
Mukai (\cite{MukTata}) in the case where $X$ is a K3 surface. 

Notation: we let $E^{\vee}=R\lhom(E,\OO_X)$ denote the derived dual
of an object $E$ of $D(X)$.

\section{Commuting Spherical Objects}\label{s:except}
Throughout this section we assume that $X$ is a smooth Calabi-Yau variety of
dimension $d$.
\begin{defn} An object $E$ of $D(X)$ is \strong{exceptional} if
$\Ext^i(E,E)$ is a small as possible (the precise definition depends
on $X$ but we will not need to be very definite in what follows). We say that $E$ is
\strong{spherical} if 
$\dim\Ext^i(E,E)=1$ for $i=1$ or $i=d$ and is zero otherwise. We
say that $E$ is \strong{rigid} if just $\Ext^1(E,E)=0$.
\end{defn}
Note that a simple rigid sheaf on a Calabi-Yau 2 or 3-fold is automatically
exceptional and spherical by Serre duality.
To any vector bundle $E$ we can associate a canonical (surjective) map
$E\otimes\Hom(E,\OO_x)\to\OO_x$ given by evaluation. We shall denote
the domain of such maps by $E_H$ for short and the kernel by
$E_x$. This extends to a map for any object $E$ of $D(X)$. We shall
denote a choice of cone on such a map by $F_x$. Then when $E$ is a
bundle, $E_x=F_x[1]$.  

In a ground breaking paper by Seidel and Thomas \cite{SeidelThomas} it is shown
(in somewhat greater generality) that when $E$ is a spherical object
in $D(X)$,
the family of $F_x$ give rise to a Fourier-Mukai transform $D(X)\to
D(X)$, denoted $\Phi_E$ or, more usually, $T_E$ (the \strong{spherical
  twist} associated to $E$).
The kernel
of the transform is given by the shift by $-1$ of the cone on the canonical map 
$\rhom(\pi_1^* E,\pi_2^*E)\to \OO_\Delta$
given by adjunction from the composite map
\[\xy\xymatrix@C=7ex{\pi_2^*E\ar[r]^/-8pt/{\pi_2^*E\otimes
    \rho}& {\pi_2^*E\otimes\OO_\Delta}\ar[r]^{\sim}&
  {\pi_1^*E \ltensor \OO_\Delta}}\endxy\]
 where $\pi_i:X\times X\to X$  are the two projection maps and
 $\rho:\OO_{X\times X}\to\OO_\Delta$ is the canonical restriction
 map. We shall denote the functor $\rhom(\pi_1^*(E\Ltensor
 {-}),\pi_2^*E)$ by $\Psi_E$. So for all $G\in D(X)$ we have a
 triangle
\[\Phi_E(G)\to\Psi_E(G)\to G\]
which is natural in $G$ (rather unusually for triangles of functors).
Their proof that these do give
 Fourier-Mukai transforms is fairly direct although a somewhat  
more elegant proof was later given by Ploog (\cite{PloogThesis}) using a clever
choice of spanning class (see \cite{HuyBook} for further details). In this
paper, we shall give yet another less elegant but more elementary
proof in the spirit of Mukai's original paper (\cite{MukTata}). 

The main point of the \cite{SeidelThomas} paper was to show that certain
families of spherical objects give rise to a representation of the
Braid group on the derived category. As a corollary of the key
computational result they also show that if $E$ and $F$ are two
spherical objects such that $\Hom(E,F[i])=0$ for all $i$ then their FM
transforms commute. We can generalise this a little as follows.
\begin{defn}
We call a finite collection $E_i$, $1\leq i\leq n$ of 
objects of $D(X)$ \strong{strongly spherical} if
\begin{equation}
\dim\Hom(E_i,E_j[k])=\begin{cases} 1& \text{if }i=j\text{ and
    (}k=0\text{ or }k=d\text{)}\\
0& \text{otherwise}\end{cases}
\end{equation}
In other words, each of the numbers $\dim\Hom(E_i,E_j[k])$  are
as small as possible.
\end{defn}
Then for a strongly spherical collection $\Gamma=\{E_i\}_{i=1}^n$ we have a
finite cone (in the sense of limits) $E_i\boxtimes
E_i^{\vee}\to\OO_\Delta$. This has a limit (up to shift) constructed explicitly as
the cone on $\bigoplus_{i=1}^n E_i\boxtimes
E_{i}^{\vee}\to\OO_\Delta$. Denote the limit by $E_{1,2,\ldots,n}$ and
its associated integral transform by $\Phi_\Gamma$. Then the
following is an easy exercise
\begin{prop}[\cite{SeidelThomas}] For any strongly spherical
  collection $\Gamma$ of objects on a
  Calabi-Yau $d$-fold,
$\Phi_\Gamma=\Phi_{E_1}\compo\Phi_{E_{2}}\compo\cdots\compo\Phi_{E_{n}}$
\end{prop}

In fact, there is a  converse:
\begin{thm} \label{t:commute}
Suppose $E$ and $F$ are two spherical objects 
in $D(X)$ such that $\Phi_E$ and $\Phi_F$ are distinct.
Then $\Phi_E\compo\Phi_F\cong\Phi_F\compo\Phi_E$ implies that $F\in E^\perp$.
\end{thm}

Before proving this we prove a technical lemma first proposed by David
Ploog in his thesis (\cite[Question 1.23]{PloogThesis}). We let
$\langle E\rangle$ denote the smallest triangulated category
containing $E$ in $D(X)$. This means that each object has a filtration
whose factors are all shifts of isomorphic copies of $E$.
\begin{lemma}\label{l:convploog}
Suppose $E$ is a spherical object of $D(X)$ and $d=\dim X\geq 2$. Then, for any object
$G\in D(X)$, $\Phi_E(G)=G[-d]$ if and only if $G\in\langle E\rangle$.
\end{lemma}
\begin{proof}
Recall that $G\in E^\perp$ if and only if $\Phi_E(G)=G$
(see \cite[Lemma 1.22]{PloogThesis}).
The reverse implication of our lemma was also proved in \cite[Lemma 1.22]{PloogThesis}. So suppose
$\Phi_E(G)=G[-d]$. Define
\[d_E(G)=\sum_{i=-\infty}^{\infty}\dim\Hom(E,G[i]).\]
We induct on $d_E(G)$. If $d_E(G)=0$ then $G\in E^\perp$ and
so $\Phi_E(G)=G$ and hence $G=0$. If $d_E(G)=1$ (wlog $\Hom(E,G)\neq0$) then $G[-d]$ fits in a
triangle
\[G[-d]\buildrel f^\vee\over\longrightarrow E\buildrel f\over\longrightarrow G,\]
where the unique maps (up to scale) are Serre dual to each other. But
then $f\compo f^\vee:G[-d]\to G$ must be Serre dual to the identity
$G\to G$ and so cannot vanish. But $f\compo f^\vee=0$ as the composite
of two consecutive maps of a triangle must always vanish. The
contradiction shows that $d_E(G)$ cannot equal $1$. Now assume that for
all $n<d_E(G)$ we know that if $d_E(G')=n$ and $\Phi_E(G')=G'[-d]$ then
$G'\in\langle E\rangle$. Pick any $f\in\bigoplus\Hom(E,G[i])$ and again
without loss of generality assume $i=0$. Let $C$ be a cone on $f:E\to
G$. Then $\Phi_E(C)=C[-d]$ because $\Phi_E(f)=f[-d]$. But we also
have that $d_E(C)=d_E(G)-2$ by applying $\Hom(E,-)$ to the triangle
defining $C$ and because $\dim X>1$. Then by induction $d_E(G)$ must be
even and $C\in\langle E\rangle$. Hence, $G\in\langle E\rangle$ as it
is an extension of $C$ by $E$.
\end{proof}
\begin{remark}
We can extract a bit more from the proof by observing
that it shows that if $G\in\langle E\rangle$ has $d_E(G)=2$ then $G\cong E[i]$ for
some integer $i$. In fact, we can go further to observe that
$d_E(G)/2$ is the length of a filtration of $G\in\langle E\rangle$
with factors given by shifts of $E$ (always under the
assumption that $d>1$). It follows that the length of such a
filtration is well defined as a function of $G$.
 
We shall use this in the following way: if $F\in\langle E\rangle$ is
spherical then applying
$F[i]\to$ to the triangle $F[-d]\to \Psi_E(F)\to F$ implies that
$d_E(F)=2$ and so $F\cong E[i]$ for some integer $i$. 
\end{remark}
\begin{lemma}\label{l:swap}
Suppose $E$ and $F$ are two spherical objects such that $\Phi_E$ and
$\Phi_F$ commute. Then $G\in\langle E\rangle$ if and only if
$\Phi_F(G)\in\langle E\rangle$.
\end{lemma}
\begin{proof}
For any $G\in\langle E\rangle$ we have
\[\Phi_E(\Phi_F(G))\cong\Phi_F(\Phi_E(G))\cong\Phi_F(G[-d]).\]
So $\Phi_F(G)\in\langle E\rangle$ by Lemma \ref{l:convploog}. Applying this to
$G=\Phi^{-1}_F(G')$ gives us the converse as well.
\end{proof}

\begin{proof}[Proof of Theorem \ref{t:commute}.] 
Assume that that $\Phi_E$ and $\Phi_F$ commute and suppose that $E$
and $F$ are not orthogonal.
Then $\Phi_E(F)\in\langle F\rangle$ by Lemma \ref{l:swap}. But
$\Phi_E(F)$ is spherical and so by the remark above, $\Phi_E(F)=F[i]$
for some $i$. By assumption, we have
a non-zero map $E\to F$ (replacing $F$ by a suitable shift if necessary).
Applying the composite functor $\Phi^n_E[nd]$, for any positive integer
$n$ to this gives a non-zero map
$E\to F[n(i+d)]$. But $D(X)$ has bounded cohomology and
exts and so $i=-d$. 


So $\Phi_E(F)\in\langle E\rangle$ by Lemma \ref{l:convploog} again. Then
$F\in\langle E\rangle$ by Lemma \ref{l:swap}. By the remark,
$F=E[i]$ for some $i$ and that implies that $\Phi_E=\Phi_F$
contradicting our assumption.
\end{proof}

\section{Spherical Bundles}
We shall now restrict our attention to the case of spherical
bundles on complex Calabi-Yau $d$-folds. We shall see that this case can be tackled more directly in
the spirit of Mukai's paper. 

We first assume that $E$ is a simple rigid bundle and
consider the double
exact complex associated to the bi-functor $\Ext^*(-,-)$ applied to
the short exact sequence\footnote{The
reader is urged to write
a large part of this double complex out on a large piece of paper
before proceeding!}
\begin{center}$0\to E_x\to E_H\to\OO_x\to0.$
\end{center}
Using the fact that $\Ext^i(E_H,\OO_x)=0$ for all $i>0$ and
$\Ext^i(\OO_x,E_H)$ vanishes for all $i<d$, we have
$\dim\Ext^1(\OO_x,F)=1$, $\dim\Ext^1(F,\OO_x)=\rk(E)^2-1+d$,
$\dim\Hom(F,E_H)=\dim\Hom(E_H,E_H)=\rk(E)^2$  and,
crucially, $\Ext^1(F,E_H)=0$ (using the fact that $d>2$ for this: the
case $d=2$ is much simpler and is left to the reader).
From this we have 
\[\dim\Ext^1(F,F)=d-1+\dim\Hom(F,F)\]
Since $\Ext^1(E_H,\OO_x)=0$, we have that the map
\[\Ext^2(\OO_x,F)\to\Ext^2(E,F)\]
vanishes and so
$\Ext^2(F,F)\to\Ext^2(\OO_x,F)$ surjects.
The map \[\Hom(F,F)\to\Ext^1(\OO_x,F)\cong\C\] is the boundary map and
must be non-zero as the identity map is contained in the
domain. Hence, this map also surjects. We can conclude
\[\dim\Hom(F,F)=\dim\Hom(E_H,F)+1\]
The following result is a stronger version of \cite{MukTata}, Prop 3.9.
\begin{lemma} The map $\Hom(E_H,E_H)\to\Hom(E_H,\OO_x)$ injects
\end{lemma}
\begin{proof}
Consider a map $f:E_H\to E_H$. If we fix a basis for $\Hom(E,\OO_x)$,
then $f$ is given by an $r\times r$ matrix with scalar entries (since
$E$ is simple). The image of $f$ is given by a subspace $V$ of
$\Hom(E,\OO_x)$ and $f$ is zero if and only if this subspace is
zero. But if it is not zero then the image of $E\otimes V$ in $\OO_x$
is non-zero and so the image of $f$ in $\Hom(E_H,\OO_x)$ is also
non-zero.
\end{proof}
We deduce that $\Hom(E_H,F)=0$ and hence $\dim\Hom(F,F)=1$.
Now we can conclude that $\dim\Ext^1(F,F)=d$. 

Next we consider two distinct points $x$ and $y$ of $X$ and the two
associated kernels $F_x$ and $F_y$. Since $\Ext^i(\OO_x,\OO_y)=0$ for
all $i$ and $F_y$ is locally-free away from $x$ we can conclude from
the double exact sequence associated to the two sequences for $F_x$
and $F_y$, that $\Hom(F_x,F_y)\cong\Hom(E_H,F_y)=0$ and
$\Ext^1(F_x,F_y)\cong\Ext^1(E_H,F_y)$ which is also zero.

The following generalises Corollary 2.12 of \cite{MukTata}.
\begin{prop}\label{p:exti}
If $E$ is a simple rigid vector bundle and $d>3$ then there are
natural isomorphisms
\[\Ext^i(F_x,F_y)\cong\Ext^i(\OO_x,\OO_y)\oplus\Ext^i(E_H,E_H)\]
for all $x,y\in X$ (not necessarily distinct) and $1<i<d-1$.
\end{prop}
\begin{proof}
The proof uses the double exact sequence
we considered above. Start at $i=2$ and observe that
$\Ext^n(E_H,F_y)\cong\Ext^n(E_H,E_H)$ for $1\leq n<d$ (the case
$n=1$ follows because $E$ is rigid) and there is a
natural injection of $\Ext^n(E_H,E_H)$ into $\Ext^n(F_x,E_H)$. We also
have $\Ext^n(F_x,\OO_y)\cong\Ext^{n+1}(\OO_x,\OO_y)$ and so the map
$g:\Ext^n(F_x,E_H)\to\Ext^n(F_x,\OO_y)$ is given by the composite 
\[\Ext^n(F_x,E_H)\to\Ext^{n+1}(\OO_x,E_H)\to\Ext^{n+1}(\OO_x,\OO_y)\to\Ext^n(F_x,E_H),\]
But $\Ext^{n+1}(\OO_x,E_H)=0$ and so
the composite vanishes for $n=1,\dots,d-1$. Moreover, the surjection
$\Ext^n(F_x,F_y)\to\Ext^n(F_x,E_H)$ splits naturally since the image
is \[\Ext^n(E_H,E_H)\cong\Ext^n(E_H,F_y)\] and the image of this in
$\Ext^n(F_x,F_y)\to\Ext^n(E_H,E_H)$ is the identity.
\end{proof}

This shows that $\{F_y\}$ is a strongly simple family. Using Theorem
\ref{t:TB}, we have an alternative proof of
\begin{thm}[\cite{SeidelThomas}, \cite{PloogThesis}] If $E$ is an spherical bundle on a Calabi-Yau $d$-fold
$X$ then the moduli space of sheaves $\{F_x\}$ constructed above is
naturally isomorphic to $X$ and gives rise to a non-trivial Fourier-Mukai
transform $D(X)\to D(X)$.
\end{thm}

\section{Recovering the Strongly Spherical Collection}
\label{s:recover}
We shall now consider the reverse process: given a Fourier-Mukai
transform determined by a family of non-locally-free torsion-free
sheaves $\{F_y\}$ with dimension 0 singularity sets, can we
find a strongly spherical collection of bundles
$\Gamma=\{E_i\}_{i=0}^n$ such that $F_x$ is the kernel of the
canonical map 
$\bigoplus_{i=0}^nE_i\otimes\Hom(E_i,\OO_x)\to \OO_x$? We shall see that this is indeed
possible. The first observation we need to make is 
that the parameter space $\{F_y\}$ is naturally (isomorphic to) $X$. This is immediate
since the map $F_y\to F_y^{**}$ has quotient $\OO_T$ and we see that
the parameter space $Y$ sits inside a space of kernels
$F_y^{**}\to\OO_T$ as $T$ varies in $\Hilb^{|T|}(X)$.
Since the moduli space must be
complete we see that the map $Y\to X$ given by the singularity of
$F_y$ is an isomorphism. We also see that $F_y^{**}=F_{y'}^{**}$ for
any pair $y$ and $y'$. We shall write $F$ for $F_y^{**}$. Since $F$ is
locally-free away from $x$ and from $y$ we see that $F$ is
locally-free over the whole of $X$. Without loss of generality we
assume in what follows that the isomorphism $Y\cong X$ is the identity.

Using the double exact sequence from the previous section we can
immediately conclude that
$\dim\Hom(F,F)=\rk(F)$ and $\Hom(F,F)\cong\Hom(F,\OO_x),$
for any $x\in X$. We can also conclude that $\Ext^i(F,F)=0$ for $i=1,\dots,d-1$.
If $\rk(F)=1$ then $F$ must be exceptional. Assume now that
$\rk(F)>1$. We observe also that the kernels of a suitable family of
maps $\lambda_x:F\to\OO_x$, as $x$ 
varies, generate the family $\{F_x\}$. Since $\dim\Hom(F,F)>1$ we can
find an endomorphism of $F$ which has rank less than $r$ and so we have a
sheaf $P$ which factors such an endomorphism. We can assume $P$ is
reflexive by factoring the torsion out of $F/P=Q$, say. We now consider the
double exact sequences associated to pairs of short exact sequences
taken from 
\[\begin{split}0\to F_x\to &\;F\to\OO_x\to 0,\\ 
0\to P\to &\; F\to Q\to 0 \text{\qquad and}\\
0\to K\to &\;F\to P\to 0.
\end{split}
\]
From these it follows that $\Hom(P,F_x)=0$ and
$\Hom(Q,F_x)=0$. It follows from this that $\Ext^1(Q,F_x)=0$ and, crucially,
$\Hom(Q,F)=\Hom(Q,\OO_x)$ and $\Hom(P,F)=\Hom(P,\OO_x)$. These imply
that both $P$ and $Q$ are locally-free. 

We now appeal to the following useful technical result (true in much
greater generality for suitable objects in any noetherian abelian category).
\begin{lemma} If $E$ is a torsion-free sheaf which is not simple then
there exists a simple sheaf $G$ (not necessarily unique) and an injection
$\alpha:G\hookrightarrow E$ and a surjection $\beta:E\twoheadrightarrow G$ such
that either $\beta\alpha$ is zero or the identity. Moreover, if $G\to
E$ is any non-zero map then it must inject.
\end{lemma}
\begin{proof}
Since $E$ is not simple, we can consider the set of sheaves $G$ which
factor non-isomorphisms $E\to E$. Such a sheaf $G$ is automatically
torsion-free and gives rise to maps $\alpha$ and $\beta$. The set is
partially ordered by  compositions $E\twoheadrightarrow
G\twoheadrightarrow G'\hookrightarrow E$. Since 
$r(G')<r(G)$ (otherwise the kernel of $G\to G'$ would be a torsion
sheaf), we can pick (using Zorn's Lemma) a minimal element with respect to this order. Call
it $G$. Then $G$ is simple since otherwise we could factor a map $G\to
G$ via $G'$ which would be strictly smaller than $G$ in the order. Now
the composite $\beta\alpha$ is either zero or a multiple of the
identity (in which case we replace $\beta$ with a suitable multiple).

The last statement follows because if such a map is not injective then
the image would be strictly smaller in the order.
\end{proof}

Applying this to our current situation we may assume $P$ is simple and
is minimal with respect to the ordering of the proof above. 
Moreover, any (non-zero) map $P\to F$ must inject. 

We now repeat this construction in a family. Suppose, as in the
previous section, that $\EE$ is the universal sheaf corresponding to
the family $\{F_y\}$ and consider $S=\EE^{**}/\EE$. Since, $F_y$ is
singular only at $y$ we have that $S|_{X\times\{y\}}=\OO_y$ and so
(wlog) $S$ is supported on the diagonal $\Delta\subset X\times X$ and is
locally-free there. If we twist by $\pi_2^*(\pi_{2*}S)^*$ then we may
assume without further loss of generality that $S=\OO_\Delta$.

Observe that $\EE^{**}$ is flat over both
projections and has the property that $\EE^{**}|_{X\times\{y\}}=F$ for
all $y\in X$ and so is locally-free. Observe we have a diagram of
natural transformations of functors 
\[\Phi_\EE\lra\Phi_{\EE^{**}}\lra \Id\]
This diagram has the
property that for any object $G\in D(X)$ there is a distinguished
triangle
\begin{equation}\Phi_\EE G\lra\Phi_{\EE^{**}}G\lra G.\label{eq-tri}\end{equation}
which is natural in $G$. Since $\Phi_\EE F=F[-d]$ we see that
$F\to \Phi_\EE F[1]$ is zero and so
$\Phi_{\EE^{**}}(F^*)\cong F^*\oplus
F^*[-d]$. Hence, $\EE^{**}|_{\{x\}\times X}\cong F^*$. 

\begin{lemma} In the given situation, $\Phi^0_{\EE^{**}}(P^*)\cong
P^*$.
\end{lemma}
\begin{proof}
By the semi-continuity of direct images
$\Phi^0_{\EE^{**}}(P^*)$ is locally-free of rank 
$r(P)$. We also have $\Hom(P,F_y)=0=\Ext^1(P,F_y)$ and so $\Phi^0_\EE(P^*)=0=\Phi^1_{\EE}(P^*)$.
The the cohomology of the triangle~(\ref{eq-tri}) provides the
required isomorphism.
\end{proof}

If we use the Leray-Serre spectral sequence for $\pi_2$ we see that
\begin{align*}
H^0((P^*\boxtimes
P)\otimes\EE^{**})&\cong H^0(R^0\pi_{2*}(\pi_1^*P^*\otimes\EE^{**})\otimes P)\\
&\cong H^0(R^0\Phi_{\EE^{**}}(P^*)\otimes P)\\
&\cong H^0(P^*\otimes P).\end{align*}
So we have natural isomorphisms $H^0((P^*\boxtimes
P)\otimes\EE)\cong\Hom(P^*,P^*)\cong\C\langle\id\rangle$
and dually we also have $H^0((P\boxtimes
P^*)\otimes\EE)\cong\Hom(P^*,P^*)$. 
We can conclude that there are unique maps (up to scalars) $\alpha:P\boxtimes
P^*\to\EE$ and $\beta:\EE\to P\boxtimes P^*$. If we apply
$R^0\pi_{2*}\compo (P^*\boxtimes P)\otimes(-)$ to these maps we obtain the
maps $\alpha'$ and $\beta': P^*\otimes P\to P^*\otimes P$. But
$\alpha'|_\OO$ has image $\OO$ and $\beta'$ is non-zero on this copy
of $\OO$ (corresponding to the identity element in $P^*\otimes
P$). Hence, $\beta'\compo\alpha'$ is not zero and so
$\beta\compo\alpha$ is also not zero. But $P$ is simple and thus
$P\boxtimes P^*$ is also simple (using the Leray-Serre spectral sequence
again). Consequently, $\beta\compo\alpha$ is the identity map. This
implies that $\EE^{**}=(P\boxtimes P^*)\oplus Q$ for some vector
bundle $Q$. It also follows that $P$ is spherical as it is a direct
summand of $F$. 

But now, $Q$ enjoys the same properties as $\EE^{**}$ and again we can
choose a simple $P'$ such that $Q=(P'{}^*\boxtimes P')\oplus Q'$. Repeating,
we have $\EE^{**}=\bigoplus_{i=1}^n \EE_i$, where $\EE_i\cong P_i\boxtimes
P_i^*$ and $P_i$ are spherical bundles. Observe that the uniqueness
of $\alpha$ and $\beta$ imply that $\Ext^k(P_i,P_j)=0$ for all $k$ and
$i\neq j$. 

We have thus proved:
\begin{thm} Let $X$ and $Y$ be (smooth) Calabi-Yau $d$-folds. If
$\mathbb{F}\to X\times Y$ is a family of properly 
torsion-free sheaves over $X$ parametrized by $Y$ with 0-dimensional
singularity sets and $\Phi_\mathbb{F}$ is a Fourier-Mukai Transform then 
\begin{enumerate}
\item there is a isomorphism $\phi:Y\to X$ and
\item there exists a unique strongly spherical collection of bundles
  $\Gamma=\{P_i\}_{i=0}^n$ on $X$ such that $(1\times\phi)^*\Phi_\Gamma=\Phi_F$.
\end{enumerate}
\end{thm}

In the case of a K3 surface, if $\Pic X=\ZZ\langle h\rangle$ then
strongly spherical collections can only have cardinality $1$. This can
be easily seen from the numerical invariants of such a collection. In
that case, we recover Yoshioka's result (\cite{YoshAb}) that a family of properly torsion-free
sheaves giving rise to an FM transform arise from a spherical
object. But in general, this will not be the case. For example, if $L$
is a line bundle on a K3 surface whose sheaf cohomology vanishes in
every degree then $\{\OO_X,L\}$ is a strongly spherical collection.


\section*{Acknowledgements} 
The author would like to thank Will Donovan for useful comments and
Tom Bridgeland and Richard Thomas for several helpful 
suggestions on an early draft of the paper.

\bibliographystyle{amsplain}
\bibliography{p54}

\end{document}